\documentclass[letter,11pt,reqno]{amsart}

\usepackage[utf8]{inputenc}

\usepackage{etex}

\usepackage{xcolor}

\definecolor{verydarkblue}{rgb}{0,0,0.5}

\usepackage{aliascnt}
\usepackage[
    breaklinks,
    colorlinks,
    citecolor=verydarkblue,
    linkcolor=verydarkblue,
    urlcolor=verydarkblue,
    pagebackref=true,
    hyperindex
]{hyperref}

\backrefenglish

\usepackage{fancyhdr}

\usepackage[
    hscale=0.7,
    vscale=0.75,
    headheight=13pt,
    centering,
]{geometry}

\usepackage{amsmath}
\usepackage{amsthm}
\usepackage{amssymb}
\usepackage{mathtools}
\usepackage{bm}		%for boldface greek symbols
\usepackage{mathdots}
\usepackage{framed}
\usepackage[capitalize]{cleveref}
\usepackage{array}
\usepackage[alphabetic,msc-links]{amsrefs}
\usepackage[all,cmtip]{xy}

\theoremstyle{plain}

\newtheorem{introtheorem}{Theorem}

\crefname{introtheorem}{Theorem}{Theorems}

\newaliascnt{introcorollary}{introtheorem}
\newtheorem{introcorollary}[introcorollary]{Corollary}
\aliascntresetthe{introcorollary}
\crefname{introcorollary}{Corollary}{Corollaries}

\newtheorem{theorem}{Theorem}[section]
\crefname{theorem}{Theorem}{Theorems}

\newaliascnt{proposition}{theorem}
\newtheorem{proposition}[proposition]{Proposition}
\aliascntresetthe{proposition}
\crefname{proposition}{Proposition}{Propositions}

\newaliascnt{lemma}{theorem}
\newtheorem{lemma}[lemma]{Lemma}
\aliascntresetthe{lemma}
\crefname{lemma}{Lemma}{Lemmas}

\newaliascnt{corollary}{theorem}
\newtheorem{corollary}[corollary]{Corollary}
\aliascntresetthe{corollary}
\crefname{corollary}{Corollary}{Corollaries}

\newaliascnt{problem}{theorem}

\aliascntresetthe{problem}
\crefname{problem}{Problem}{Problems}

\newaliascnt{conjecture}{theorem}

\aliascntresetthe{conjecture}
\crefname{conjecture}{Conjecture}{Conjectures}

\newaliascnt{claim}{theorem}

\aliascntresetthe{claim}
\crefname{claim}{Claim}{Claims}

\theoremstyle{definition}

\newaliascnt{definition}{theorem}
\newtheorem{definition}[definition]{Definition}
\aliascntresetthe{definition}
\crefname{definition}{Definition}{Definitions}

\newaliascnt{notation}{theorem}

\aliascntresetthe{notation}
\crefname{notation}{Notation}{Notations}

\newaliascnt{situation}{theorem}

\aliascntresetthe{situation}
\crefname{situation}{Situation}{Situations}

\theoremstyle{remark}

\newaliascnt{remark}{theorem}
\newtheorem{remark}[remark]{Remark}
\aliascntresetthe{remark}
\crefname{remark}{Remark}{Remarks}

\newaliascnt{example}{theorem}
\newtheorem{example}[example]{Example}
\aliascntresetthe{example}
\crefname{example}{Example}{Examples}

\newaliascnt{question}{theorem}

\aliascntresetthe{question}
\crefname{question}{Question}{Questions}

\crefname{step}{Step}{Steps}

\numberwithin{figure}{section}

\numberwithin{equation}{section}
\usepackage{marginnote}

\def\Z{{\mathbb Z}}
\def\Q{{\mathbb Q}}

\def\C{{\mathbb C}}

\def\L{{\mathbb L}}

\def\A{{\mathbb A}}
\def\P{{\mathbb P}}

\def\cF{\mathcal{F}}

\def\cM{\mathcal{M}}

\def\I{\mathcal{I}}

\def\O{\mathcal{O}}

\def\fm{\mathfrak{m}}
\def\fn{\mathfrak{n}}

\def\a{\alpha}
\def\b{\beta}

\def\f{\varphi}
\def\ff{\psi}
\def\e{\eta}

\def\k{k}
\def\n{\nu}
\def\m{\mu}
\def\om{\omega}

\def\s{\sigma}
\def\t{\tau}

\def\Om{\Omega}

\def\.{\cdot}
\let\circum\^
\def\^{\widehat}
\def\~{\widetilde}
\def\o{\circ}

\def\rat{\dashrightarrow}
\def\surj{\twoheadrightarrow}

\def\({\left(}
\def\){\right)}

\def\*{{}^*}

\renewcommand{\and}{ \ \ \text{ and } \ \ }

\def\reg{\mathrm{reg}}
\def\sing{\mathrm{sing}}

\def\tor{\mathrm{tor}}
\def\Jac{\mathrm{Jac}}

\DeclareMathOperator{\im} {Im}

\DeclareMathOperator{\Spec} {Spec}

\DeclareMathOperator{\ord} {ord}

\DeclareMathOperator{\Fitt} {Fitt}

\DeclareMathOperator{\Bs} {Bs}

\DeclareMathOperator{\Tor} {Tor}

\DeclareMathOperator{\diag} {diag}
\DeclareMathOperator{\Var} {Var}
\DeclareMathOperator{\EP}{EP}

\def\embdim{\mathrm{edim}}

\begin{document}

\title{Relative Mather discrepancy on arc spaces}

%\dedicatory{Personal notes}

\author{Tommaso de Fernex}
\address{Department of Mathematics, University of Utah, 155 South 1400 East,
Salt Lake City, UT 84112, USA}
\email{{\tt defernex@math.utah.edu}}

\author{Zach Mere}
\address{Department of Mathematics, University of Utah, 155 South 1400 East,
Salt Lake City, UT 84112, USA}
\email{{\tt zachmere@math.utah.edu}}

\subjclass[2020]{%
Primary {\scriptsize 14E18};
Secondary {\scriptsize 14B05}.}
\keywords{Motivic integration, Nash blow-up, jet scheme, $K$-equivalence.}

\thanks{%
The research of the first author was partially supported by NSF grant DMS-2001254.
The research of the second author was partially supported by NSF RTG DMS-1840190.
}

\begin{abstract}
Given any generically \'etale morphism of varieties $f \colon X \to Y$, 
we define the relative Mather discrepancy function on the arc space $X_\infty$
of the domain and show that this function computes
the dimension of the kernel of the differential map of the induced morphism 
on arc spaces $f_\infty \colon X_\infty \to Y_\infty$.
We relate this result to the change-of-variable formula in motivic integration.
We introduce the notion of $\^K$-equivalence, which agrees with $K$-equivalence for smooth varieties, and 
prove that arc spaces of $\^K$-equivalent varieties of arbitrary
characteristic have the same motivic volume. 
\end{abstract}

\maketitle

\section{Introduction}

Let $X$ be a variety over a field $k$. Its arc space, denoted by $X_\infty$, 
parameterizes formal arcs $\a \colon \Spec k_\a[[t]] \to X$, where $k_\a$
is the residue field of $\a$ viewed as a point of $X_\infty$. 

A formula for the sheaf of differentials of $X_\infty$ was introduced in \cite{dFD20}
to study properties of local rings in $X_\infty$
and their connections to the singularities of $X$. 
Further studies in this direction were carried out in \cite{CdFD22,CdFD24}. 
A key ingredient in these studies is the analysis of the differential of the map on arc spaces 
$f_\infty \colon X_\infty \to Y_\infty$ induced by a generically \'etale morphism
of varieties $f \colon X \to Y$. Given a point $\a \in X_\infty$ and setting $\b = f_\infty(\a) \in Y_\infty$, 
one looks at the differential
\[
df_{\infty,\a} \colon (\Om_{Y_\infty/k} \otimes \k_\b) \otimes k_\a  
\to \Om_{X_\infty/k} \otimes \k_\a.
\]

Interestingly, as long as $\a$ is not fully contained within the 
support of $\Om_{X/Y}$, this map is always surjective. The question is about its kernel. 
In good situations (e.g., if either $X$ or $Y$ are smooth), its dimension
can be computed in terms of the order of vanishing of suitable Jacobian ideals along the given arcs, 
but in general all prior results on this 
only give upper or lower bounds.

The main goal of this paper is to compute the kernel of $df_{\infty,\a}$. 
Using a construction closely related to Mather's approach to Chern classes on singular varieties
(cf.\ \cite{Mat70,Mac74}) along the lines of \cite{dFEI08},
we introduce a function on the space of arcs $X_\infty$
of $X$ which we call the \emph{relative Mather discrepancy function},
or \emph{Mather ramification function}, depending on the context
(e.g., if $f$ is birational, or finite).

This function, denoted by $\ord_{\^K_{X/Y}}$, 
is defined on $X_\infty \setminus (X_\sing)_\infty$
by taking the Nash blow-up $\n \colon \^X \to X$ of $X$, 
looking at the natural map
\[
\f \colon \n^*f^*\Om_Y^d \otimes \O_{\^X}(-1) \to \n^*\Om_X^d \otimes \O_{\^X}(-1) \surj \O_{\^X}
\]
where $\O_{\^X}(1) := (\n^*\Om_X^d)/_\tor$ is the tautological line bundle of the Nash blow-up, 
and setting
\[
\ord_{\^K_{X/Y}}(\a) := \ord_{\~\a}(\im(\f))
\]
for $\a \in X_\infty \setminus (X_\sing)_\infty$, 
where $\~\a \in \^X_\infty$ is the unique lift of $\a$.
The notation is suggested by the situation where $f \o \n$ factors through 
the Nash blow-up $\m \colon \^Y \to Y$ via a morphism $\^f \colon \^X \to \^Y$, 
in which case we have $\im(\f) = \O_{\^X}(-\^K_{X/Y})$
for an effective Cartier divisor $\^K_{X/Y} \in |\O_{\^X}(1) \otimes \^f^*\O_{\^Y}(-1)|$
(cf.\ \cref{l:^K}). 

By construction, $\ord_{\^K_{X/Y}}$ agrees with 
the order of the relative canonical divisor $K_{X/Y}$
if both $X$ and $Y$ are nonsingular, and with 
the order of the Jacobian ideal $\Jac_f \subset \O_X$ of $f$ if $X$ is nonsingular.
If $Y$ is nonsingular and $X$ is allowed to be singular, then we have 
$\ord_{\^K_{X/Y}} = \ord_{\Jac_f} - \ord_{\Jac_X}$.
In general, if both $X$ and $Y$ are singular, then the function $\ord_{\^K_{X/Y}}$ provides a new 
measure of discrepancy (or ramification) along $f$.

We can now state our main result. 

\begin{introtheorem}[cf.\ \cref{t:df-infty}]
\label{t-intro:df-infty}
Let $f \colon X \to Y$ be a generically \'etale morphism of varieties over a field $k$.
Let $\a \in X_\infty$ and $\b = f_\infty(\a) \in Y_\infty$. 
Assume that $X$ is nonsingular at the generic point $\a(\e)$ of the arc 
and $f$ is \'etale in a neighborhood of $\a(\e)$.
Then the differential map 
$df_{\infty,\a} \colon (\Om_{Y_\infty/k} \otimes k_\b) \otimes k_\a 
\to \Om_{X_\infty/k} \otimes k_\a$
is surjective with kernel of dimension
\[
\dim(\ker (df_{\infty,\a})) = \ord_{\^K_{X/Y}}(\a). 
\]
\end{introtheorem}

As a corollary of this result, assuming the ground field is perfect,
we obtain the following formula comparing embedding dimensions of the corresponding
local rings.

\begin{introcorollary}[cf.\ \cref{c:Tf-infty}]
\label{c-intro:df-infty}
With the same assumptions as in \cref{t-intro:df-infty},
if $k$ is perfect, then we have
\[
\embdim(\O_{Y_\infty,\b}) = \embdim(\O_{X_\infty,\a}) + \ord_{\^K_{X/Y}}(\a).
\]
\end{introcorollary}

\cref{t-intro:df-infty} (or, more precisely, \cref{c-intro:df-infty}) is closely related to the change-of-variables formula in motivic integration (see \cref{r:relation-change-of-vars}). Originally introduced by Kontsevich for proper birational morphisms between smooth varieties and extended in \cite{DL99} to the setting of resolutions of singularities, the formula has since been generalized in several directions. Most recently, it was established for possibly singular Artin stacks in \cite{SU23} by reducing to the smooth case previously treated in \cite{SU21}, and for wild Deligne–Mumford stacks in \cite{Yas24}, where the approach does not rely on resolution of singularities.

Specializing to the case of varieties, the formula takes the form:
\[
\int_{Y_\infty} \L^{-\ord_Z} d\m = 
\int_{X_\infty} \L^{-\ord_{f^{-1}Z} - \ord_{\^K_{X/Y}}} d\m
\]
where $f \colon X \to Y$ is a proper birational morphism of varieties 
over a field $k$ and $Z \subset Y$ is a proper closed subscheme.
This identity follows from \cite[Theorem~1.1]{Yas24} upon restricting to varieties, once one observes that the relative Mather discrepancy function coincides with the \emph{order of the Jacobian} function introduced in \cite{CLNS18} (cf.\ \cref{r:ordjac}).

Although the equality in \cite{Yas24} is not stated in the usual Grothendieck ring 
$\^\cM_k$ due to additional conditions required in the stack-theoretic setting, the proof, when restricted to varieties, yields the formula in the classical motivic framework. 
In fact, Takehiko Yasuda has informed us of a draft version of his forthcoming book \cite{Yas21}, currently available on his homepage, where motivic integration is developed in full generality for arbitrary varieties.

Motivic integration was introduced to prove Batyrev's conjecture that
birationally equivalent complex Calabi--Yau manifolds have the same Hodge numbers \cite{Kon95,Bat99}.
Kontsevich's proof goes by showing that the two Calabi--Yau manifolds define the same class in the motivic ring
$\^\cM_\C$. 
The motivic change-of-variables formula yields an extension of this latter property to
birationally equivalent Calabi--Yau manifolds in positive characteristics. 
Here, the term \emph{Calabi--Yau} is used in a weak sense, 
with no restriction on the ground field and only requiring that 
the variety is smooth and projective and its canonical sheaf is trivial.

In characteristic zero, birationally equivalent Calabi--Yau manifolds
are examples of {$K$-equivalent} manifolds, namely, 
smooth proper varieties $Y$ and $Y'$ 
admitting a common resolution $X$ such that $K_{X/Y} = K_{X/Y'}$. 
To extend this definition to positive characteristics, we introduce the notion of \emph{$\^K$-equivalence}, where
two birationally equivalent proper varieties $Y$ and $Y'$ are said to be {$\^K$-equivalent}
if for every (equivalently, for some) variety $X$ with proper birational morphisms
$f \colon X \to Y$ and $f' \colon X \to Y'$ inducing morphisms
$\^f \colon \^X \to \^Y$ and $\^f' \colon \^X \to \^Y'$ on the Nash blow-ups, we have $\^K_{X/Y} = \^K_{X/Y'}$.
This notion does not require any assumption on singularities, and 
agrees with the notion of $K$-equivalence for nonsingular varieties.

\begin{introtheorem}[cf.\ \cref{t:^K-equiv}]
\label{t-intro:^K-equiv}
Arc spaces of $\^K$-equivalent varieties over a field $k$ have the same motivic volume. In 
particular, if the varieties are smooth, then they
define the same class in the motivic ring $\^\cM_k$. 
\end{introtheorem}

This implies, in particular, that birationally equivalent Calabi--Yau manifolds
define the same class in $\^\cM_k$ (cf.\ \cref{c:bir-CY}). This answers
a question posed to the first author by Dori Bejleri. The same result was also obtained in \cite{Yas21}. 

It is interesting to compare our approach with the one followed in 
\cite{Yas21}, where the definition of $K$-equivalence is extended to 
positive characteristics in the context of log pairs using crepant modifications. 
The two approaches are easily seen to lead to equivalent definitions
for smooth varieties.
In \cite{Yas21}, it is proved that $K$-equivalent smooth projective varieties 
define the same class in the motivic ring, and one can regard
\cref{t-intro:^K-equiv} as an extension of that result to all varieties. 
One should notice, however, that the same property does not extend to singular $K$-equivalent varieties 
(cf.\ \cref{r:flopping-contr}). 
In this regard, $\^K$-equivalence appears to be a suitable notion 
to extend the property to arbitrary varieties. 

Restricting to the smooth case, one deduces from 
\cref{t-intro:^K-equiv} that $K$-equivalent smooth varieties have the same 
same Betti numbers and \'etale Euler characteristic, 
a property that was also observed in \cite{Yas21}. 

It remains unclear, however, whether one should expect that in positive characteristic 
$K$-equivalent manifolds have the same Hodge numbers, as happens in characteristic zero. 
In characteristic zero, it is conjectured that $K$-equivalent manifolds
are $D$-equivalent (i.e., have equivalent bounded derived categories of complexes of coherent sheaves)
and $D$-equivalent manifolds have the same Hodge numbers \cite{Kaw02,Orl05}. 
However, the latter property is known to fail in positive characteristics \cite{AB23}.
This pathology, 
along with the lack of any good strategy for defining in positive characteristics an Euler function 
on the Grothendieck ring of varieties $K_0(\Var_k)$
which sends the class of a nonsingular projective variety to its Hodge polynomial 
(even assuming resolution of singularities, weak factorization, and 
a generalization of Bittner's theorem \cite{Bit04}), 
suggest that the situation may be more pathological in positive characteristics.

\subsection*{Acknowledgements}

We wish to thank Dori Bejleri for asking the question about Calabi--Yau manifolds, Christopher Chiu, 
Roi Docampo, and Will Legg for several discussions on related topics,  
and Mircea Musta\c t\u a for useful discussions and for pointing us to \cite{AB23}.
We are grateful to Takehiko Yasuda 
for useful comments and for bringing \cite{Yas24,Yas21} to our attention. 
Finally, we thank the referee for valuable comments and suggestions.

\section{Notation}

We work over a field $k$.
A \emph{variety} is a generically smooth integral separated scheme over $k$. 
The Jacobian ideal of a variety $X$ is defined by $\Jac_X := \Fitt^d(\Om_{X/k})$
where $d = \dim X$. If $f \colon X \to Y$ is a generically finite morphism
of varieties of the same dimension, then its Jacobian ideal is defined by $\Jac_f := \Fitt^0(\Om_{X/Y})$.

Given a scheme $X$ over $k$, we denote by $X_\infty$ its arc space. 
For a morphism of $k$-schemes $f \colon X \to Y$, 
we denote by $f_\infty \colon X_\infty \to Y_\infty$
the induced morphism on arc spaces. We refer
to \cite{EM09,CLNS18} for definitions and general properties. 

For an arc $\a$ on $X$, we denote by $\a$ both
the point $\a \colon \Spec k_\a \to X_\infty$ and the map
$\a \colon \Spec k_\a[[t]] \to X$.
Any arc $\a \in X$ defines a semi-valuation 
$\ord_\a \colon \O_{X,\a(0)} \to \Z_{\ge 0} \cup \{\infty\}$ given by $\ord_t \o \a^\sharp$. 
For an ideal sheaf $\I \subset \O_X$, we denote by 
$\ord_\I \colon X_\infty \to \Z \cup \{\infty\}$ the function 
given by $\ord_\I(\a) := \ord_\a(\I)$. 
If $Z \subset X$ is a closed subscheme, then we also denote by $\ord_Z$ 
the order function defined by its ideal sheaf $\I_Z \subset \O_X$.

\section{Relative Mather discrepancy}
\label{s:Mather}

In this section we extend the definition of relative Mather canonical divisor
originally given in \cite{dFEI08} in the case resolutions of singularities,
and define a relative Mather discrepancy function on arc spaces
for generically \'etale morphisms of varieties. 

Let $X$ be a variety over a field $k$, and set $d = \dim X$.

\begin{definition}
A proper birational morphism $\n \colon \^X \to X$ is a \emph{Nash blow-up} of $X$
if $\n^*\Om_X$ admits a locally free quotient of rank $d$. 
If $\n$ is universal with this property, then we refer to it as \emph{the} Nash blow-up of $X$.
We denote $\O_{\^X}(1) := (\n^*\Om_{X/k}^d)/_\tor$.
\end{definition}

\begin{remark}
The Nash blow-up of $X$ is unique up to isomorphism, constructed as the closure
of $X_\reg$ in $\P(\Om_{X/k}^d)$, and $\O_{\^X}(1)$ is a relatively ample line bundle over $X$.
\end{remark}

Let $f \colon X \to Y$ be a generically \'etale morphism of varieties of dimension $d$.
Let $\n \colon \^X \to X$ and $\m \colon \^Y \to Y$ be the Nash blow-ups of $X$ and $Y$. 
If $f \o \n$ factors through $\m$, 
then let $\^f \colon \^Y \to \^X$ be the unique morphism such that $\m \o \^f := f \o \n$. 

\begin{definition}
\label{d:Nash-ideal}
The ideal sheaf
$\fn_{X/Y} \subset \O_{\^X}$ defined by the image of the natural map 
\[
\n^* f^*\Om_{Y/k}^d \otimes \O_{\^X}(-1) \to \n^*\Om_{X/k}^d \otimes \O_{\^X}(-1) \surj \O_{\^X}
\]
is called the \emph{relative Nash ideal} of $f$.
\end{definition}

\begin{lemma}
\label{l:^K}
Assume that $f \o \n$ factors through $\m$. Then there is an effective Cartier divisor 
$\^K_{X/Y} \in |\O_{\^X}(1) \otimes \^f^*\O_{\^Y}(-1)|$
such that $\fn_{X/Y} = \O_{\^X}(-\^K_{X/Y})$. 
\end{lemma}

\begin{proof}
Using the surjection $\m^*\Om_{Y/k}^d \surj \O_{\^Y}(1)$ and the factorization 
$f \o \n =\m \o \^f$, we obtain the 
commutative diagram
\[
\xymatrix{
\n^* f^*\Om_{Y/k}^d \otimes \O_{\^X}(-1) \ar[r] \ar@{->>}[d] 
& \n^*\Om_{X/k}^d \otimes \O_{\^X}(-1) \ar@{->>}[d] 
\\
\^f^*\O_{\^Y}(1) \otimes \O_{\^X}(-1) \ar[r]
&\O_{\^X} 
}
\]
and the statement follows by comparing images.
\end{proof}

\begin{definition}
In the setting of \cref{l:^K}, 
we call $\^K_{X/Y}$ the \emph{relative Mather canonical divisor},
or \emph{Mather ramification divisor}, of $X$ over $Y$.
The first term is typically used when $f$ is birational and the latter when $f$ is finite. 
\end{definition}

\begin{remark}
\label{r:^K-exceptional}
By construction, $\^K_{X/Y}$ is supported within the locus where
$f \o \n$ is not locally an isomorphism. In particular, if $f$ is birational and $Y$ is normal, then 
$\^K_{X/Y}$ is exceptional over $Y$, i.e., $(f\o \n)_* \^K_{X/Y} = 0$. 
\end{remark}

\begin{remark}
The definition of $\^K_{X/Y}$ agrees with the one given in \cite{dFEI08} when $f$ is 
a resolution of singularities. 
See \cite[Remark~1.5]{dFEI08} for an explanation of the terminology.
\end{remark}

The following property is immediate from the definition. 

\begin{lemma}
\label{l:^K-additive}
If $g \colon Y \to Z$ is another generically \'etale morphism of varieties and 
$\m \o g$ factors through the Nash blow-up of $Z$, then $\^K_{X/Z} = \^K_{X/Y} + \^f^*\^K_{Y/Z}$. 
\end{lemma}

\begin{remark}
It is worthwhile to point out that Jacobian ideals do not satisfy a similar additivity property,
not even up to integral closure. For example, suppose that 
$f \colon X \to Y$ is a generically \'etale morphism where $Y$ is a closed subvariety of
an affine space $\A^N_k$, and $y \in Y$ is a singular point. Let $v$ be a divisorial valuation on $X$
centered at $y$ on $Y$, and let $g \colon Y \to Z$ be the 
finite morphism induced by a general linear projection $\A^N_k \to Z = \A^d_k$
where $d = \dim X = \dim Y$. 
If the projection is general with respect to $v$, then $v(\Jac_f) = v(\Jac_{g\o f})$. On the other hand, 
$\Jac_g$ is non-trivial at $y$, hence $v(\Jac_g) > 0$. It follows that
$v(\Jac_{g\o f}) \ne v(\Jac_g) + v(\Jac_f)$. 
\end{remark}

In general, $\^K_{X/Y}$ is not defined if $\n\o f$ does not factor through $\m$, 
and we need to work with the ideal sheaf $\fn_{X/Y}$ instead. 
With slight abuse of notation, we still use the symbol $\^K_{X/Y}$ in the notation introduced 
in the next definition.

\begin{definition}
Let $f \colon X \to Y$ be a generically \'etale morphism of varieties.
We define the \emph{relative Mather discrepancy function}
(or \emph{Mather ramification function}) of $X$ over $Y$ to be the function
\[
\ord_{\^K_{X/Y}} \colon X_\infty \setminus (X_\sing)_\infty \to \Z
\] 
given by setting
\[
\ord_{\^K_{X/Y}}(\a) := \ord_{\~\a}(\fn_{X/Y})
\]
where $\~\a \in \^X_\infty$ is the unique lift of $\a$, which exists by the properness of $\n$
and the fact that the indeterminacy locus $\Bs(\n^{-1})$ of $\n^{-1}$ is 
contained in the singular locus $X_\sing$ of $X$.
\end{definition}

\begin{remark}
\label{r:ordjac}
The relative Mather discrepancy function defined above agrees with the \emph{order of the Jacobian}
function introduced in \cite[Chapter~5, Section~3.1]{CLNS18}
and used in \cite{Yas24,Yas21}. 
We prefer to keep a clear distinction, in our terminology, 
between this function and the function 
$\ord_{\Jac_f}$ defined by the order of the Jacobian
ideal of $f$, which is defined in full generality and typically differs from 
$\ord_{\^K_{X/Y}}$ if $X$ is not smooth.
\end{remark}

\begin{remark}
A word of caution on the choice of notation involving the symbol $\^K_{X/Y}$. 
As already remarked, the divisor $\^K_{X/Y}$ is only defined if $\n \o f$ factors through $\m$. Even then,
$\^K_{X/Y}$ is not a divisor on $X$ but on $\^X$. 
In good situations, one can ensure that $\n \o f$ factors through $\m$ by either
taking $f$ to factor through $\m$, or, when that is not possible,
by replacing $\n$ with a `higher' Nash blow-up of $X$. 
The resulting function agrees with the above definition as long as $\Bs(\n^{-1})$
remains contained in $X_\sing$. 
If there is a change in the base locus of $\n^{-1}$,
then the functions agree away from the arc space of the larger base locus, 
which has infinite codimension in $X_\infty$.
\end{remark}

\section{Differential maps on arc spaces}

This section is devoted to the following result,
which relates the relative Mather discrepancy function
to the differential map of the morphism induced at the arc level.

\begin{theorem}
\label{t:df-infty}
Let $f \colon X \to Y$ be a generically \'etale morphism of varieties
over a field $k$.
\begin{enumerate}
\item
\label{i:2a}
For every arc $\a \in X_\infty$ such that $X$ is nonsingular at $\a(\e)$ and $f$ is  
\'etale in a neighborhood of $\a(\e)$, the differential map 
\[
df_{\infty,\a} \colon \Om_{Y_\infty/k} \otimes \k_\a \to \Om_{X_\infty/k} \otimes \k_\a
\]
is surjective with kernel of dimension equal to $\ord_{\^K_{X/Y}}(\a)$. 
\item
\label{i:1}
As functions on $X_\infty \setminus (X_\sing)_\infty$, we have:
\begin{itemize}
\item
$\ord_{\^K_{X/Y}} = \ord_{K_{X/Y}}$ if both $X$ and $Y$ are nonsingular; 
\item
$\ord_{\^K_{X/Y}} = \ord_{\Jac_f}$ if $X$ is nonsingular; 
\item
$\ord_{\^K_{X/Y}} = \ord_{\Jac_f} - \ord_{\Jac_X}$ if $Y$ is nonsingular.
\end{itemize}
In general, we have
\[
\ord_{\Jac_f} - \ord_{\Jac_X} \le \ord_{\^K_{X/Y}} \le \min\{\ord_{\Jac_f}, \ord_{\Jac_f} - \ord_{\Jac_X} + \ord_{\Jac_Y} \}.
\] 
\end{enumerate}
\end{theorem}

\begin{proof}
For short, we denote $K := k_\a$ and 
\[
P_\infty := \frac{K(\hskip-1pt(t)\hskip-1pt)}{t K[[t]]},
\]
which we regard both as a module over $\O_{X_\infty}$ via $\a$
and as a module over $\O_{Y_\infty}$ via $f_\infty \o \a$.
Let $\b := f_\infty(\a)$; note that $k_\b \subset K$ via $f_\infty$.

Since $X$ is nonsingular at $\a(\e)$, we have  $\a \not\in (X_\sing)_\infty$,
hence $\ord_{\^K_{X/Y}}(\a)$ is defined.
To compute this number and relate it to $df_{\infty,\a}$, we use the isomorphisms 
\[
\Om_{X_\infty/k} \otimes k_\a \cong \Om_{X/k} \otimes P_\infty \cong P_\infty^d, 
\quad
\Om_{Y_\infty/k} \otimes k_\a  \cong \Om_{Y/k} \otimes P_\infty \cong P_\infty^d, 
\]
given by \cite{dFD20}, where $d = \dim X = \dim Y$.
We will compute these spaces in two steps, by first pulling back
along $\a$ and $f_\infty\o \a$ respectively (which corresponds to tensoring by $K[[t]]$),
and then tensoring by $P_\infty$ over $K[[t]]$. Note that $P_\infty$ is a $t$-divisible
$K[[t]]$-module, hence tensoring by it kills $K[[t]]$-torsion. 

We start from the exact sequence 
\[ 
f^*\Om_{Y/k} \to \Om_{X/k} \to \Om_{X/Y} \to 0.
\]
By pulling back along $\a$, which corresponds to tensoring by $k_\a[[t]]$, we
obtain the top exact row in the commutative diagram
\begin{equation}
\label{eq:1}
\xymatrix{
K[[t]]^d \oplus \bigoplus \frac{K[t]}{(t^{b_i})} \ar[r]^-\f 
&
K[[t]]^d \oplus \bigoplus \frac{K[t]}{(t^{a_i})} \ar[r]^-\ff \ar@{->>}[d] 
&
\bigoplus \frac{K[t]}{(t^{c_i})} \ar[r]
&
0
\\
K[[t]]^d \ar@{^(->}[u] \ar[r]^A
&
K[[t]]^d
&
&
}
\end{equation}
where the vertical arrows are the natural inclusion and projection
and $A \in K[[t]]^{d\times d}$ is the matrix giving 
the bottom $K[[t]]$-linear map. 
Here $d = \dim Y$ and $a_i$ (resp., $b_i$, $c_i$) are the invariant factors
of $\Om_{X/k}$ (resp., $\Om_{Y/k}$, $\Om_{X/Y}$) (cf.\ \cite[Section~6]{dFD20}), hence
\[
a := \sum a_i = \ord_\a(\Jac_X),
\quad
b := \sum b_i = \ord_\a(\Jac_Y),
\quad
c := \sum c_i = \ord_\a(\Jac_f).
\]
These are all finite orders since 
$X$ is nonsingular and $f$ is \'etale in a neighborhood of $\a(\e)$.
This last fact also implies that there is no free summand in the term
given by the pull back of $\Om_{X/Y}$, since $\a(\e)$ is
not contained in the support of this sheaf. 

Tensoring everything by $K(\hskip-1pt(t)\hskip-1pt)$, the vertical arrows become isomorphisms
and the bottom arrow gives a $K(\hskip-1pt(t)\hskip-1pt)$-linear map
\[
K(\hskip-1pt(t)\hskip-1pt)^d \xrightarrow{A} K(\hskip-1pt(t)\hskip-1pt)^d.
\]
We interpret this as the map obtained by pulling back the morphism $f^*\Om_{Y/k} \to \Om_{X/k}$
along the generic point $\a_\e \colon \Spec K(\hskip-1pt(t)\hskip-1pt) \to X$ of the arc. 
This map is an isomorphism because $f$ is \'etale in a neighborhood of the image of $\a_\e$. 
Therefore, we obtain a short exact sequence
\[
0 \to K[[t]]^d 
\xrightarrow{A} K[[t]]^d \to \bigoplus \frac{K[t]}{(t^{e_i})} \to 0.
\]
The matrix $A$ is the same as the one before, and we now see that it
is has coefficients in $K[[t]]$. Its determinant has order 
\[
\ord_t(\det A) = e := \sum e_i.
\] 

The differential map $df_{\infty,\a}$ is obtained from $\f$
by tensoring the sequence \eqref{eq:1} by $P_\infty$, which gives the exact sequence
\[
P_\infty^d \xrightarrow{A} P_\infty^d \to 0.
\]
This shows that $df_{\infty,\a}$ is surjective, with kernel of dimension 
\[
\dim_K(\ker(df_{\infty,\a})) = e.
\]
While not strictly necessary, the computation of the dimension of the kernel
becomes particularly transparent after diagonalizing $A$ over $k_\a[[t]]$. Without 
loss of generality, we may assume that bases have been chosen such that 
$A = \diag(t^{e_1},\dots,t^{e_d})$ with $\sum e_i = e$, and the computation reduces to computing 
\[
\dim_K\big(\Tor_1^{K[[t]]}\big(\tfrac{K[t]}{(t^{e_i})},P_\infty \big) \big) = e_i
\]
using the resolution $K[[t]] \xrightarrow{t^{e_i}} K[[t]]$. 

The next step is to compute $e$. 
By construction, we have a commutative diagram
\[
\xymatrix{
\a^*f^* \Om_{Y/k} \ar[r] \ar@{->>}[d]
& \a^* \Om_{X/k} \ar@{->>}[d]
\\
K[[t]]^d \ar[r]^A 
& K[[t]]^d 
}
\]
where the vertical arrows are given by killing the $t$-torsion.
Taking top exterior powers
and composing with the natural map $\n^*\Om_{X/k}^d \surj \O_{\^X}(1)$, we obtain the commutative diagram
\[
\xymatrix{
\a^*f^* \Om_{Y/k}^d \ar[r]^\lambda \ar@{->>}[d]^\s
& \a^* \Om_{X/k}^d \ar@{->>}[r]^-\m \ar@{->>}[d]^\t
& \~\a^* \O_{\^X}(1) \ar@{->>}[dl]^{\t'}
\\
K[[t]] \ar[r]^{\det A}
& K[[t]] 
& 
}.
\]
Here we use the fact that taking exterior powers preserves subjectivity and 
commutes with pullbacks. 
The fact that $\t$ factors through $\t'$ follows from the observation
that $\m$ is just the quotient map killing the $t$-torsion of $\a^* \Om_{X/k}^d$, and since 
$\~\a^* \O_{\^X}(1) \cong K[[t]]$ and $\t'$ is surjective, it follows that $\t'$
is an isomorphism. 
As $\im(\m \o \lambda) = \~\a^* ( \fn_{X/Y}\.\O_{\^X}(1))$, we conclude that 
\[
e = \ord_{\^K_{X/Y}}(\a).
\]
This proves property \eqref{i:2a} stated in the theorem.

Regarding \eqref{i:1}, first note that $\Jac_f = \O_X(-K_{X/Y})$ if both $X$ and $Y$ are nonsingular.
Thus, we are left to prove that:
\begin{itemize}
\item
$e = c$ if $a = 0$ (i.e., if $X$ is nonsingular at $\a(0)$);
\item
$e = c-a$ if $b = 0$ (i.e., if $Y$ is nonsingular at $\b(0)$);
\item
$c-a \le e \le \min\{c, c-a+b\}$ in general.
\end{itemize}
By construction, $e \le c$. The other estimates follow by looking at the short exact sequence
\[
0 \to
K[[t]]^d \oplus \bigoplus \frac{K[t]}{(t^{b_i'})} 
\to K[[t]]^d \oplus \bigoplus \frac{K[t]}{(t^{a_i})} 
\to \bigoplus \frac{K[t]}{(t^{c_i})} \to 0,
\]
obtained from \eqref{eq:1} by killing the kernel of $\f$, which is necessarily contained in the torsion 
summand $\bigoplus \frac{K[t]}{(t^{b_i})}$. Note that, by construction, $\sum b_i' \le \sum b_i$. 
Taking Tor into $P_\infty$, we obtain a long exact sequence of the form
\[
0 \to \bigoplus \frac{K[t]}{(t^{b_i'})} 
\to \bigoplus \frac{K[t]}{(t^{a_i})} 
\to \bigoplus \frac{K[t]}{(t^{c_i})} 
\to P_\infty^d 
\xrightarrow{A} P_\infty^d
\to 0,
\]
and the estimates on $e$ easily follow from here. 
\end{proof}

\begin{corollary}
\label{c:Tf-infty}
With the same assumptions as in \cref{t:df-infty}, 
assume that $k$ is perfect. Then the cotangent map 
\[
T_\a^*f_\infty \colon \fm_\b/\fm_\b^2 \otimes k_\a \to \fm_\a/\fm_\a^2 
\]
is surjective with kernel of dimension equal to $\ord_{\^K_{X/Y}}(\a)$. In particular, 
the embedding dimensions of the local rings at $\a$ and $\b$ satisfy the formula
\[
\embdim(\O_{Y_\infty,\b}) = \embdim(\O_{X_\infty,\a}) + \ord_{\^K_{X/Y}}(\a).
\]
\end{corollary}

\begin{proof}
This follows from \cref{t:df-infty} and the snake lemma applied to the following
commutative diagram, which relates the cotangent map $T_\a^*f_\infty$ to 
the differential map $df_{\infty,\a}$:
\[
\xymatrix{
0 \ar[r] 
& \fm_\b/\fm_\b^2 \otimes k_\a \ar[r] \ar[d]^{T_\a^*f_\infty} 
& \Om_{Y_\infty/k} \otimes k_\a \ar[r] \ar[d]^{df_{\infty,\a}} 
& \Om_{k_\b/k} \ar[r] \ar[d]^{\Theta_\a}
& 0
\\
0 \ar[r] 
& \fm_\a/\fm_\a^2 \ar[r] 
& \Om_{X_\infty/k} \otimes k_\a \ar[r] 
& \Om_{k_\a/k} \ar[r] 
& 0
}
\]
The exactness of the rows follows from \cite[Theorem~25.2]{Mat89}; the assumption that $k$
is perfect is needed to ensure that the extensions $k_\a/k$ and $k_\b/k$ are separable, 
hence 0-smooth \cite[Theorems~26.3 and~26.9]{Mat89}.
\cite[Theorem~4.3]{CdFD24} implies that 
the field extension $k_\b \subset k_\a$ induced by $f_\infty$
is finite separable, hence $\Theta_\a$ is injective by \cite[Theorem~26.6]{Mat89}.
Then the snake lemma implies that the kernel of $T_\a^*f_\infty$ 
coincides with the kernel of $df_{\infty,\a}$, hence has dimension 
equal to $\ord_{\^K_{X/Y}}(\a)$ by \cref{t:df-infty}.
On the other hand, the cokernel of $T_\a^*f_\infty$ 
is contained in the cokernel of $df_{\infty,\a}$, which is zero by \cref{t:df-infty}, 
hence $T_\a^*f_\infty$ is surjective. 
\end{proof}

\begin{remark}
\label{r:relation-change-of-vars}
The last formula in \cref{c:Tf-infty} is closely related to the change-of-variables formula
in motivic integration \cite{DL99,Yas21}. 
The precise relation is understood using the interpretation of embedding dimension 
of arc spaces in terms of jet codimension, which is established in \cite{dFD20},
and the formula computing the fibers of the restrition to the set of liftable jets of the 
maps induced on jet schemes $f_m \colon X_m \to Y_m$, which lies at the core
of the proof of the motivic change-of-variable formula. 
\cref{c:Tf-infty} recovers the dimension of these fibers, but not their motivic
incarnation (i.e., the fact that these fibers are affine spaces).
\end{remark}

\section{$\^K$-equivalence}

In this section we introduce the notion of $\^K$-equivalence
and use motivic integration to show that arc spaces of $\^K$-equivalent varieties
have the same motivic volume.  

We begin by recalling the definition of $K$-equivalence in characteristic zero. 

\begin{definition}
Two birationally equivalent nonsingular proper varieties $Y$ and $Y'$ of characteristic zero
are said to be \emph{$K$-equivalent} if for every common resolution of singularities
$X$ (i.e., a nonsingular variety $X$ with proper birational morphisms $f \colon X \to Y$
and $f' \colon X \to Y'$), we have $K_{X/Y} = K_{X/Y'}$. 
\end{definition}

It is easy to see that it suffices to check the property for one common resolution. 
Also, it is a well-known consequence of the Negativity Lemma \cite[Lemma~3.39]{KM98} that
it suffices to require that $K_{X/Y} \equiv K_{X/Y'}$. 
The assumption on characteristic is needed to ensure the existence of a common resolution.

Here, we extend the definition in two directions, by allowing the varieties
to be arbitrarily singular and the characteristic to be positive,
without assuming the existence of resolutions of singularities.
We use the relative Mather canonical divisor introduced in \cref{s:Mather}.

\begin{definition}
\label{d:^K-equiv}
Two birationally equivalent proper varieties $Y$ and $Y'$ are said to be \emph{$\^K$-equivalent}
if for every variety $X$ with proper birational morphisms $f \colon X \to Y$
and $f' \colon X \to Y'$ inducing morphisms on Nash blow-ups
$\^f \colon \^X \to \^Y$ and $\^f' \colon \^X \to \^Y'$, we have $\^K_{X/Y} = \^K_{X/Y'}$.
\end{definition}

\begin{lemma}
To see whether two birationally equivalent proper varieties $Y$ and $Y'$
are $\^K$-equivalent, it suffices to check 
the defining property for a single model $X$.
\end{lemma}

\begin{proof}
Let $X$ be a variety as in the definition 
such that $\^K_{X/Y} = \^K_{X/Y}$.
Suppose $X'$ is another variety with proper birational morphisms $g \colon X' \to Y$
and $g' \colon X' \to Y'$ inducing morphisms on Nash blow-ups
$\^g \colon \^X' \to \^Y$ and $\^g' \colon \^X' \to \^Y'$.
Pick any variety $Z$ with proper birational morphisms $h \colon Z \to X$
and $h' \colon Z \to X'$ inducing morphisms on Nash blow-ups
$\^h \colon \^Z \to \^X$ and $\^h' \colon \^Z \to \^X'$. 
Using the additivity of relative Mather canonical divisors 
along compositions of proper birational morphisms (see \cref{l:^K-additive}), 
we see that $\^K_{X/Y} = \^K_{X/Y'}$ if and only if $\^K_{X'/Y} = \^K_{X'/Y'}$.
\end{proof}

\begin{remark}
If $Y$ and $Y'$ are nonsingular, then it follows by \cref{t:df-infty} that 
$\^K$-equivalence is equivalent to the condition
that there exists a variety $X$ with proper birational morphisms
$f \colon X \to Y$ and $f' \colon X \to Y'$ as in the definition, such that the Jacobian ideals
$\Jac_f$ and $\Jac_{f'}$ have the same integral closure in $\O_X$.
\end{remark}

\begin{remark}
In \cite{Yas21}, the definition of $K$-equivalence is extended to arbitrary characteristics
in the context of log pairs, by requiring the existence of common crepant model.
It is immediate to check that if $Y$ and $Y'$ are smooth, then
$K$-equivalence and $\^K$-equivalence are equivalent notions.
The definition of $K$-equivalence extends to normal varieties 
as long as their canonical classes are $\Q$-Cartier, but
in this generality it is no longer equivalent to the definition of $\^K$-equivalence. 
\end{remark}

We denote by $K_0(\Var_k)$ the \emph{Grothendieck ring of varieties} over $k$. 
This is given by the free abelian group generated by 
isomorphism classes of separated schemes of finite type
over $k$ modulo the relations $[X] = [X \setminus Z] + [Z]$ whenever
$Z \subset X$ is a closed subscheme, and $[X] + [Y] = [X \sqcup Y]$, with
product given by $[X]\.[Y] := [X \times Y]$. 
We set $\cM_k := K_0(\Var_k)[\L^{-1}]$
where $\L = [\A^1_k]$ is the class of the affine line. 
The \emph{motivic ring} $\^\cM_k$ is defined as the completion of $\cM_k$
with respect to the filtration $\cF^n \subset \^\cM_k$
generated, as a subgroup, by classes $[Y] \L^{-m}$
with $\dim Y - m \le -n$. 

As explained in \cite[Chapter~2, Section~(3.5.9)]{CLNS18},
there exists a function
\[
\EP \colon K_0(\Var_k) \to \Z[t]
\]
whose value on the class of a variety $X$ only depends on the isomorphism class of the variety.
The value of this function on the class of a smooth proper variety $X$ over $k$ is equal to
the \emph{Euler--Poincar\'e polynomial}
\[
\EP(X) = \sum_{n=0}^{2d} (-1)^n \dim H^n(X_{k^s},\Q_\ell)t^n
\]
where $d$ is the dimension of the variety, 
$k^s$ is a separable closure of $k$, and $\ell$ is a prime number. 
Furthermore, the function descends 
to a function defined on the image of $\cM_k$ in the motivic ring $\^\cM_k$
\cite[Chapter~2, Example~4.3.6]{CLNS18}.

\begin{theorem}
\label{t:^K-equiv}
If $Y$ and $Y'$ are $\^K$-equivalent varieties over a field $k$, then 
their arc spaces have the same motivic volume 
$\int_{Y_\infty} d\m = \int_{Y'_\infty} d\m$. 
In particular, if $Y$ and $Y'$ are smooth, then they define
the same motivic class $[Y] = [Y']$ in the motivic ring $\^\cM_k$, and
hence have the same Euler--Poincar\'e polynomial $\EP(Y) = \EP(Y')$.
\end{theorem}

\begin{proof}
Given a model $X$ as in the definition of $\^K$-equivalence, 
the motivic change-of-variables formula established in \cite{Yas24,Yas21} gives
\[
\int_{Y_\infty} d\m 
= \int_{X_\infty} \L^{-\ord_{\^K_{X/Y}}}d\m
= \int_{X_\infty} \L^{-\ord_{\^K_{X/Y'}}}d\m
= \int_{Y'_\infty} d\m, 
\]
and these are equal to $[Y]$ and $[Y']$ is the varieties are smooth. 
The last statement follows by the aforementioned results of \cite{CLNS18}.
\end{proof}

\begin{remark}
\label{r:flopping-contr}
A similar result is obtained in \cite{Yas21} assuming $Y$ and $Y'$ are smooth. 
One can regard \cref{t:^K-equiv} as a generalization of Yasuda's result to singular varieties. 
The same property does not hold, however, for $K$-equivalent varieties
if one allows singularities.
For instance, if $f \colon Y \to Y'$ is a crepant resolution, then $Y$ and $Y'$ are $K$-equivalent but 
in general we may have 
$\int_{Y_\infty} d\m \ne \int_{Y'_\infty} d\m$ and $[Y] \ne [Y']$ in $\^\cM_k$
(see \cref{e:ODP} for an explicit computation).
Note that in the case of a crepant resolution $Y \to Y'$, the two varieties 
$Y$ and $Y'$ are not $\^K$-equivalent unless $Y'$ is already smooth.
\end{remark}

\begin{example}
\label{e:ODP}
Let $f \colon Y \to Y'$ be a crepant resolution of an ordinary double point of dimension $d \in \{2, 3\}$. 
Let $p$ be the singular point of $Y'$ and $E \cong \P^1$ the exceptional locus of $f$. 
Let $U = Y \setminus E$ and $U' = f(U) = Y' \setminus \{p\}$. 
In either dimension, we have 
\[
[Y] = [U] + [E] = [U'] + [p] + \L = [Y'] + \L \ne [Y'].
\]
To compare motivic integrals, 
let $g \colon X \to Y$ be the blow-up of $E$ (an isomorphism if $d = 2$), $h = f \o g \colon X \to Y'$, 
and $F$ the exceptional divisor of $h$. 
Using that $Y'$ is locally complete intersection
and observing that $\ord_F(\Jac_{Y'}) = 1$, we see that
$\^K_{X/Y'} = K_{X/Y'} + F$ (cf.\ \cite[Section~3]{dFD14}). 
Note, on the other hand, that $K_{X/Y'} = K_{X/Y}$, since $K_{Y/Y'} = 0$. 
Using these facts, we compute
\[
\int_{Y_\infty} d\m = \int_{X_\infty} \L^{-\ord_{K_{X/Y}}} d\m 
= \int_{X_\infty} \L^{-\ord_{(\^K_{X/Y'}-F)}} d\m 
\ne \int_{X_\infty} \L^{-\ord_{\^K_{X/Y'}}} d\m
= \int_{Y'_\infty} d\m.
\]
\end{example}

For the purpose of this paper, 
a \emph{Calabi--Yau manifold} is a smooth projective variety $Y$ over a field $k$ such that $\om_Y \cong \O_Y$. 
The following property is well-known in characteristic zero
and is discussed in \cite{Yas21} in connection to $K$-equivalence. 

\begin{proposition}
\label{p:CY-J-equiv}
Birationally equivalent Calabi--Yau manifolds are $\^K$-equivalent. 
\end{proposition}

\begin{proof}
The proof is analogous to that of the corresponding 
statement about $K$-equivalence in characteristic zero. 
Let $Y$ and $Y'$ be two birationally equivalent Calabi--Yau manifolds and let
$X$ be a normal 
projective variety with proper birational maps $f \colon X \to Y$ and $f' \colon X \to Y'$
(e.g., take $X$ to be the normalization of the closure of the graph
of the birational map $Y \rat Y'$ in $Y \times Y'$).
Let $\n \colon \^X \to X$ be the normalized Nash blow-up of $X$
and let $\^f \colon \^X \to Y$ and $\^f' \colon \^X \to Y'$ be the induced maps. 

Since $Y$ is nonsingular, its Nash blow-up $\^Y \to Y$ is the identity, hence
$\O_{\^Y}(1) = \om_Y$. Using then our assumption that $Y$ is Calabi--Yau, 
it follows that $\O_{\^Y}(1) \cong \O_Y$. Similarly, we have $\O_{\^Y'}(1) \cong \O_{Y'}$. 
Therefore we have
\[
\O_{\^X}(\^K_{X/Y}) \cong \O_{\^X}(1) \cong \O_{\^X}(\^K_{X/Y'}),
\]
hence $\^K_{X/Y} \sim \^K_{X/Y'}$. 
Note that $\^K_{X/Y}$ is effective and $\^f$-exceptional, 
and similarly, $\^K_{X/Y'}$ is effective and $\^f'$-exceptional
(cf.\ \cref{r:^K-exceptional}).
Applying the Negativity lemma \cite[Lemma~3.39]{KM98}
to $\pm (\^K_{X/Y} - \^K_{X/Y'})$ along $\^f$ and $\^f'$, 
we conclude that $\^K_{X/Y} = \^K_{X/Y'}$. 
\end{proof}

The following corollary also appears in \cite{Yas21}.

\begin{corollary}
\label{c:bir-CY}
Any two birationally equivalent Calabi--Yau manifolds $Y$ and $Y'$ over a field $k$
define the same class $[Y] = [Y']$ in $\^\cM_k$ and
have the same Euler--Poincar\'e polynomial $\EP(Y) = \EP(Y')$. 
\end{corollary}

The following remark was suggested by Docampo. 

\begin{remark}
One may consider the class of varieties $X$ such that $\O_{\^X}(1) \cong \O_{\^X}$. 
If such an $X$ is smooth, then it is a Calabi--Yau manifold, 
hence this class of varieties can be considered as a new 
generalization of Calabi--Yau manifolds to singular varieties. 
The interest in this notion stems from the property that any two such varieties that
are birationally equivalent are automatically $\^K$-equivalent, 
and hence \cref{t:^K-equiv} applies to them. 
It should be noted that singular Calabi--Yau varieties (i.e., normal projective 
varieties with trivial canonical sheaf) do not belong to this class. 
In fact, if $\O_{\^X}(1) \cong \O_{\^X}$ holds, then the Nash blow-up 
$\n \colon \^X \to X$ must be a finite map. This is a strong restriction
on the singularities of $X$; in particular, if $X$ is singular, then 
it cannot be normal. 
An example is given by a singular curve of geometric genus one
with only nodes, or more, generally, with ordinary multiple points (i.e., 
singular points with smooth analytic branches).
It is worth noting that this property does not behave well in families; 
one can consider for instance the example of a family of smooth plane cubics
degenerating to a singular cubic.
\end{remark}

\end{document}